\documentclass[preprint,12pt,3p]{elsarticle}
\usepackage{amsmath,amsthm,amsfonts,amssymb}
\usepackage{xcolor}
\usepackage{sectsty}
\usepackage{hyperref}
\definecolor{astral}{RGB}{46,116,181}
\sectionfont{\color{astral}}
\newtheorem{theorem}{Theorem}[section]
\newtheorem{lemma}[theorem]{Lemma}
\newtheorem{corollary}[theorem]{Corollary}
\newtheorem{proposition}[theorem]{Proposition}
\newcommand{\ep}{\scriptsize\mbox{\textcircled{$\dagger$}}}
\newcommand{\core}{\scriptsize\mbox{\textcircled{\#}}}
\newtheorem{definition}[theorem]{Definition}
\newtheorem{example}[theorem]{Example}

\journal{arxiv}
\newcommand{\cnn}{\mathbb{C}^{n\times n}}
\newcommand{\cmn}{\mathbb{C}^{m\times n}}
\newcommand{\cmm}{\mathbb{C}^{m\times m}}
\newcommand{\cnm}{\mathbb{C}^{n\times m}}
\newcommand{\cn}{\mathbb{C}^{n}}

\begin{document}
\begin{frontmatter}
\title{ {\bf 
Further results on weighted core-EP inverse of matrices
}}
\author{Ratikanta Behera$^a$, Gayatri Maharana$^\dagger$ $^b$, Jajati Keshari Sahoo$^\dagger$ $^c$}

\address{ 
 $^{a}$Department of Mathematics, 
University of Central Florida, Orlando, USA.\\
\textit{E-mail}: 
\texttt{ratikanta.behera@ucf.edu}
\vspace{.3cm}

 $^{\dagger}$ Department of Mathematics, 
BITS Pilani, K.K. Birla Goa Campus, Goa, India
\\\textit{E-mail\,$^b$}: \texttt{p20180028\symbol{'100}goa.bits-pilani.ac.in }\\
\textit{E-mail\,$^c$}: \texttt{jksahoo\symbol{'100}goa.bits-pilani.ac.in}\\
}

\begin{abstract}
In this paper, we introduce the notation of $E$-weighted core-EP and $F$-weighted dual core-EP inverse of matrices. We then obtain a few explicit expressions for the weighted core-EP inverse of matrices through other generalized inverses. Further, we discuss the existence of generalized weighted Moore-Penrose inverse and additive properties of the weighted core-EP inverse of matrices. In addition to these, we propose the star weighted core-EP and weighted core-EP star class of matrices for solving the system of matrix equations. We further elaborate on this theory by producing a few representation and characterization of star weighted core-EP and weighted core-EP star classes of matrices.
\end{abstract}
\end{frontmatter}
\section{Introduction}
\label{sec1}
The core and core-EP inverses of matrices have been intensively studied in recent years to solve a certain type of matrix equations \cite{baks, BakTr14}. Hence, a significant number of papers explored the characterizations of the core inverse and its applications in \cite{ baskett1969,  Kurata_2018, LiChen18, PreMo20}. A few properties of the core inverse and interconnections with different generalized inverses were discussed in \cite{baks,Kurata_2018,rakic,Wang_2014}. The core-EP inverse of matrices, introduced  by Prasad and Mohana  \cite{PrasadMo14}, have significantly impacted for square matrices. Then several characterizations of the core-EP inverse and its extension to rectangular matrices were discussed in \cite{ferreyra2018revisiting}. In this connection, the authors of \cite{gao2018representations} have discussed the weighted core-EP inverse and several representations in terms of matrix decomposition.  Further, a few characterizations and properties of the core-EP inverse with other inverses are discussed in \cite{Gao_2019, PrasadMo14, SahooBeheraCORE19,  PredragKM17}. The last literature on core-EP, weighted core-EP inverses of matrices along with its multifarious extensions \cite{MaH19, HaiTing19, mosic2019, zhou2019core}, motivate us to study and introduce $E$-weighted core-EP and $F$-weighted dual core-EP inverse of matrices.

We mention below a summary of the main points of the discussion.
\begin{enumerate}
\item[$\bullet$] The notations of $E$-weighted core-EP and $F$-weighted dual core-EP inverses are proposed. Through these definitions, the existence of generalized weighted Moore-Penrose inverse is discussed.

\item[$\bullet$] Introduce several explicit expression for the weighted core-EP inverse of matrices through other generalized inverses, like, Drazin inverses, weighted core inverse, and generalized Moore-Penrose inverses.

\item[$\bullet$] We have discussed additive properties of the $E$-weighted core-EP and $F$-weighted dual core-EP inverse of matrices.

\item[$\bullet$]  Introduce {\it star weighted core-EP} and {\it weighted core-EP star} matrices to solve the system of matrix equations. 

\item[$\bullet$] A few characterization and representation of star weighted core-EP and weighted core-EP star classes of matrices are discussed. 
\end{enumerate}

The main objective of this paper to investigate a new fruitful way for developing the relation of the weighted core-EP inverse of matrices with different generalized inverses, like, Drazin inverses, weighted core inverse, and generalized Moore-Penrose inverses. The results of these approaches will help the necessary freedom to deal with different types of inverses and flexibility to choose generalized inverses depending on applications. It is worth to mention the work of Mosi\'{c} in \cite{Mosic20} in which they have introduced the Drazin-star and star-Drazin matrices for solving some kind of system of matrix equations, very recently. The author also discussed the maximal classes of matrices for generating the  most general form of this class of matrices. It also motivates us to introduce the star weighted core-EP and weighted core-EP star class of matrices and then provide characterizations of $E$-weighted core-EP inverse and $F$-weighted dual core inverse of matrices in the form of the outer inverse of the Moore-Penrose inverse. 

The outline of the paper is as follows. We present some necessary definitions and notation in Section 2. Definition, existence, and several explicit expression for the weighted core-EP inverse of matrices are considered in Section 3. In Section 4, we discuss the new class of matrices (i.e., star weighted core-EP and weighted core-EP star) to solve the system of the matrix equation. In addition to these, we discuss a few characterizations of the new class of matrices. The work is concluded along with a few future perspective problems in Section 5.

\section{Preliminaries}
For convenience, throughout this paper, $ \mathbb{C}^{m\times n }$ stands for the set of $m \times n$ matrices over complex numbers. In addition, we assume the matrices $E$ and $F$ to be invertible, and hermitian. Further, we use the notation $R(A),~N(A)$, and $A^*$ for the range space, null space and conjugate transpose of $A \in \mathbb{C}^{m \times n}$ respectively. The index of $A\in \mathbb{C}^{n\times n }$ is the smallest non-negative integer $k$, such that $rank(A^k)=rank(A^{k+1})$, which is denoted by $ind(A)$. 
The Drazin inverse discussed in \cite{cline1980} for a rectangular matrix, however, it was introduced \cite{Drazin58} earlier in the context of associative rings and semigroups. Let $A\in \cnn$ and $k=ind(A)$. The Drazin inverse of $A$ is the unique matrix $X = A^D \in \cnn$ which satisfies the following equations
\begin{equation*}
\left(1^k\right)~XA^{k+1}=A^k,~~ (2)~XAX =X,~~(5)~AX= XA.
\end{equation*}
Let us recall the generalized weighted Moore-Penrose inverse \cite{sheng2007generalized, ben} of a matrix. 

\begin{definition}
Let $A\in \cmn, E\in\cmm,~F\in\mathbb{C}^{n\times n}$. A matrix $Y\in\cnm$ satisfying 
\begin{center}
 \textup{$(1)$}~$AYA = A$, \textup{$(2)$}~$YAY = Y$, \textup{$(3^E)$}~$(EAY)^* = EAY$,  \textup{$(4^F)$}~$(FYA)^* = FYA$, 
\end{center}
is called the generalized weighted Moore-Penrose inverse of $A$ and denoted by $A
^{\dagger}_{E,F}$.
\end{definition}

Note that generalized weighted Moore-Penrose inverse of a matrix does not exist always\cite{sheng2007generalized}. But the  positive definite of $E$ and $F$ leads to the existence of   $A^{\dagger}_{E,F}$. The uniqueness of   $A^{\dagger}_{E,F}$ can be verified easily.  
 Recall the definition of weighted core and dual inverse of a matrix $A$ as follows.
 \begin{definition}
 Let $A,~E\in \mathbb{C}^{n\times n}$. If a matrix $Y\in \cnn$  satisfies 
 \begin{center}
 \textup{(6)} $YA^2 = A$,~\textup{(7)} $ AY^2 = Y$,  \textup{$(3^E)$} $(EAY)^* = EAY$,
  \end{center}
 is called the $E$-weighted core inverse of $A$.
\end{definition} 
This inverse is  denoted by $A^{\core,E}$ and the uniqueness of it can found in \cite{RJR}. At the same time, the $F$-weighted  dual core inverse of $A$ denoted by $A^{F,\core}$ and defined as follows.
\begin{definition}[\cite{RJR}]
Let $A,~F\in \mathbb{C}^{n\times n}$. A matrix $Y$ is called $F$-weighted dual core inverse of $A$ if satisfies 
 \begin{center}
 \textup{(8)} $A^2Y = A$,~ \textup{(9)} $ Y^2A = Y$, and \textup{$(3^F)$} $(FYA)^* = FYA$.
 \end{center}
\end{definition}
Note that the weighted  $e$-core inverse and $f$-dual core inverse of an elements in
$*$-rings, were first introduced by Mosic et.al.\cite{mosic2018weighted}. The same authors  have defined the above two definitions though ideals and proved in Theorem 2.1 and 2.2 that, these definitions are equivalent. Now recall a few useful results from \cite{RJR} and \cite{ben}.
\begin{theorem}\label{thm2.4}(\cite{RJR})
Let  $A,~E \in \cnn$ and $ind(A)=1$. If $A\{1,3^E\}\neq \phi$, then the following five conditions are true.
\begin{enumerate}
    \item [(a)] ($A^{\core,E})^{\#}=A^2A^{\core,E}=(A^{\core,E})^{\core,E}$;
    \item[(b)] ($A^{\#})^{\core,E}=A^2A^{\core,E};$
    \item[(c)] $A^{\#}=(A^{\core,E})^2A$;
    \item[(d)] $(A^k)^{\core,E}=(A^{\core,E})^k$ for any $k\in\mathbb{N}$;
    \item[(e)] $[(A^{\core,E})^{\core,E}]^{\core,E}=A^{\core,E}$.
\end{enumerate}
\end{theorem}

\begin{lemma}[\cite{ben}]\label{lemDraz}
Let $A\in\cnn$ with $ind(A)=k$. Then $A^m$ has index 1 and $(A^m)^{\#} = (A^D)^m$ for all $m\geq k$.
\end{lemma}

\begin{lemma}[Proposition 3.2, \cite{RJR}]\label{lem2.5}
For $A\in\cnn$, if $X\in A\{6, 7\}$, then $AXA=A$ and $XAX=X$.
\end{lemma}

\begin{lemma}\label{lem27}
 Let $A,~E\in\cnn$. If $S,T\in A\{1,3^E\}$, then $AS =AT$.
 \end{lemma}
 \begin{proof}
 Let  $S,T\in A\{1,3^E\}$. Then $ASA = A =ATA$,$(EAS)^* = EAS$ and $(EAT)^* = EAT$. Using these, we obtain 
 \begin{align*}
 AS =&~ ATAS = E^{-1}EATE^{-1}EAS = E^{-1}(EAT)^*E^{-1}(EAS)^* \\
 = &~ E^{-1}(EASE^{-1}EAT)^* = E^{-1}(EASAT)^* = E^{-1}(EAT)^* = E^{-1}EAT\\
 = &~ AT. \qedhere 
\end{align*}
 \end{proof}
 Similarly, the following result follows for weighted dual core inverse.
 \begin{lemma}
 Let $A,~F\in\mathbb{C}^{n\times n}$. If $Y,Z\in A\{1,4^F\}$, then $YA =ZA$.
 \end{lemma}
\begin{lemma}[\cite{RJR}]\label{lem2.6}
Let $E\in \mathbb{C}^{n\times n}$ and  $A\in\mathbb{C}^{n\times n}$ with $ind(A)=1$. If $A\{1,3^E\}\neq \phi$ and $A^{\dagger}_{E,I}$ exist, then $A^{\core,E}=A^{\#}AA^{(1,3^E)}=A^{\#}AA^{\dagger}_{E,I}$.
\end{lemma}
 
 \begin{theorem}
Let $A,~E\in\cnn$. If  $XAX = X$ and $R(EA)=R(X^*)$, then  $AXA = A$.
 \end{theorem}
 \begin{proof}
 Let  $R(EA)=R(X^*)$. Then there exists a $V\in\cnn$ such that $EA=X^*V$. Now, 
 \begin{equation}\label{eqn3}
     EA=X^*V=X^*A^*X^*V=(AX)^*X^*V=(AX)^*EA.
 \end{equation}
From equation \eqref{eqn3}, we obtain 
\begin{equation}\label{eqqq}
    (EAX)^*=((AX^*)EAX)^*=(AX)^*EAX=EAX.
\end{equation}
Using equation \eqref{eqn3} and \eqref{eqqq}, we have 
    $A=E^{-1}(AX)^*EA=E^{-1}(EAX)^*A=AXA$.

    \end{proof}
\section{{One-sided weighted core-EP inverse}}
Motivated by the aforementioned work on ring and Banach  algebra \cite{Dijana2019, Huihui2019}, we study $E$-weighted core-EP inverse and $F$-weighted dual core-EP inverses of square matrices. Indeed, the authors Zhu and Wang in \cite{Huihui2019} defined the pseudo $e$-core invertible elements in a ring. The same authors have discussed a few characterizations of these inverses. Following the Theorem 3.2 \cite{Huihui2019}, we define matrix representation of the $E$-weighted core-EP inverse, as follows.
 \begin{definition}\label{eepdef}
 Let $A,~E\in\mathbb{C}^{n\times n}$  and $ind(A)=k$. A matrix $X$ is called the $E$-weighted core-EP inverse  of $A$ if it satisfies \begin{center}
 \textup{$(6^k)$} $XA^{k+1} = A^k$,~   \textup{$(7)$} $AX^2 = X$, and \textup{$(3^E)$} $(EAX)^* = EAX$.
 \end{center}
  \end{definition}
 The $E$-weighted core-EP inverse  of $A$ is denoted by $A^{\ep,E}$.  
 Next, we define the $F$-weighted dual core-EP inverse.
 \begin{definition}\label{fepdef}
 { Let $A,~F\in\mathbb{C}^{n\times n}$ and  $ind(A)=k$.} A matrix $X\in\cn$  satisfies \begin{center}
 \textup{$(8^k)$} $A^{k+1}X = A^k$,~   \textup{$(9)$} $X^2A = X$,  \textup{$(4^F)$} $(FXA)^* = FXA$,
 \end{center}
 is called the $F$-weighted dual core-EP inverse  of $A$ and denoted by $A^{F,\ep}$. 
 \end{definition}
  In support of the Definition \ref{eepdef} and \ref{fepdef}, we have worked out the following example.
 \begin{example}\rm
Let $A=
    \begin{pmatrix}
      4 &  3    &          0\\     
       0         &     0     &         0    \\   
      -1          &    4    &          0      \\ 
    \end{pmatrix}$,  $E=
    \begin{pmatrix}
    3 & 1 & 2\\
    1 & 1 & 1\\ 
    2 & 1 & 2\\
    \end{pmatrix}$ and $F=
    \begin{pmatrix}
 2  &    1           &   0\\       
 1     &         2     &         1\\       
 0      &        1     &2  \\  
    \end{pmatrix}$. It is easy to verify $ind(A)=2$,
   \begin{center}
    $A^{\ep,E}=\begin{pmatrix}
    5/17    &       3/34        &   3/17 \\   
       0     &         0        &      0   \\    
      -5/68   &       -3/136    &     -3/68  \\  
    \end{pmatrix}$ and $A^{F,\ep}=\begin{pmatrix}
    1/6     &       1/8       &     0 \\      
    1/9      &      1/12     &      0   \\    
    -1/18     &     -1/24   &        0    \\  
    \end{pmatrix}$.
   \end{center} 
   \end{example}
   Now, we discuss a few useful results which will be frequently used in the subsequent sections. 
  \begin{proposition}\label{lem3.2}
Let $A \in\mathbb{C}^{n\times n}$ and $ind(A)=k$. Then the following holds.
 \begin{enumerate}
     \item[(i)] If a matrix $X\in A\{\textup{$7$}\}$ then  $AX=A^mX^m$ for any $m\in\mathbb{N}$.
          \item[(ii)] If  $X\in A\{\textup{$6^k$}, \textup{$7$}\}$, then $XAX=X$ and $R(X)= R(A^k)$.
          \item[(iii)] If a matrix $X\in A\{\textup{$9$}\}$ then  $XA=X^mA^m$ for any $m\in\mathbb{N}$.
 \end{enumerate}
 \end{proposition}
 \begin{proof}
 $(i)$ Let $AX^2=X$. Then $AX=AAX^2=\cdots =A^kX^k=A^{k+1}X^{k+1}=\cdots=A^mX^m$.\\
 $(ii)$ Let $X\in A\{\textup{$6^k$}, \textup{$7$}\}$. Then by part $(i)$, we have 
 \begin{center}
     $XAX=XA^{k+1}X^{k+1}=A^{k}X^{k+1}=AX^2=X$.
 \end{center}
Further, from $X=XAX=XA^{k+1}X^{k+1}=A^kX^{k+1}$ and $A^{k}=XA^{k+1}$, we obtain $R(X)=R(A^k)$
 \end{proof}
The characterization of $E$-weighted core-EP inverse through the range condition is presented below. 
 \begin{theorem}\label{thm3.4}
Let $A,~E\in\mathbb{C}^{n\times n}$ and $ind(A)=k$. Then the following statements are equivalent.
\begin{enumerate}
    \item[(i)]  $X=A^{\ep,E}$,
    \item[(ii)] $XAX = X$,~$R(X) = R(A^k)$ and  {$R(X^*) = R(EA^k)$}.
\end{enumerate}
\end{theorem}
\begin{proof}
 {
 $(i)\Rightarrow (ii)$ By Proposition \ref{lem3.2}, it enough to show $R(X^*) = R(EA^K)$. Let $X=A^{\ep,E}$. Then 
 \begin{eqnarray*}
  X^* &=& (XAX)^* = (XE^{-1}EAX)^* = (XE^{-1}(EAX)^*)^* = (XE^{-1}(EA^kX^k)^*)^*\\
 &=& (XE^{-1}(X^k)^*(A^k)^*E)^*= EA^kX^k(XE^{-1})^*.
 \end{eqnarray*}
 Thus $ R(X^*)\subseteq R(EA^k)$. From 
 \begin{center}
      $(A^k)^*E = (EA^k)^* = (EA^kX^kA^k)^* = (EAXA^k)^* = (A^k)^*(EAX)^* = (A^k)^*EAX$,
 \end{center}
we get $EA^k=X^*A^*EA^k$ and subsequently, $R(EA^k) \subseteq  R(X^*)$. Hence  $R(X^*) = R(EA^k)$.}\\
$(ii)\Rightarrow (i)$ From $R(X) = R(A^k)$, we have $A^k = XU$, for some $U\in\cnn$. Now 
\begin{center}
   $XA^{k+1} = XAA^k = XAXU = XU = A^k$.
\end{center}
Let  {$R(EA^k)=R(X^*)$} and $R(X) = R(A^k)$. Then there exists $U,V\in\cnn$ such that $EA^k = X^*U$ and $X = A^kV$. Using these results, we obtain
\begin{equation}\label{eq1}
   (AX)^*EA^k=X^*A^*X^*U=X^*U=EA^k, \mbox { and }
\end{equation}
\begin{equation}\label{eq2}
  EAX=EAA^kV=(AX)^*EAA^kV=(AX)^*EAX=((AX)^*EAX)^*=(EAX)^*.  
\end{equation}
In view of Eqns. \eqref{eq1}, \eqref{eq2}, and invertibility of $E$, we have $A^k=AXA^k$. \\Thus 
 $AX^2 = AXA^kV = A^kV = X$. Hence $X$ is the $E$-weighted core-EP inverse of $A$.  
\end{proof}
It is worth mentioning that Zhu and Wang  \cite{Huihui2019} studied characterizations of pseudo $f$-dual core
inverses of ring. Following Theorem 3.8, \cite{Huihui2019} we study matrix representation of $F$-weighted dual core-EP inverse.
\begin{theorem}
Let $A,~F\in\mathbb{C}^{n\times n}$ and $ind(A)=k$. Then the following statements are equivalent.
\begin{enumerate}
    \item[(i)]  $X=A^{F,\ep}$,
    \item[(ii)] $XAX = X$,~$R(X^*) = R((A^k)^*)$ and $R(FX) = R((A^k)^*)$.
\end{enumerate}
\end{theorem}

We now construct the  representation of the Drazin inverse through $\{1,3^E\}$ inverse of matrix. However, Zhu and Wang \cite{Huihui2019} proved the equivalent characterization of the following result in the setting of a ring with involution (see Lemma-3.3).
\begin{lemma}\label{lem3.4}
Let $A,~E\in\cnn$ and $ind(A)=k$. For $m\geq k$, if $X=A^{\ep,E}$, then the following holds.
\begin{enumerate}
\item[(i)]  $X^m\in A^m\{1,3^E\}$.
\item[(ii)] $A^D = X^{m+1}A^m$.
\end{enumerate}
\end{lemma}
\begin{proof}
$(i)$ Let $X=A^{\ep,E}$. Then by Proposition \ref{lem3.2} $(i)$, we have 
 \begin{center}
 $A^m =XA^{m+1}=AX^2A^{m+1}=A^mX^mXA^{m+1}=A^mX^mA^m$, \mbox{ and}
   \end{center}             
 \begin{center}
    $(EA^mX^m)^* = (EAX)^* =EAX
  =EA^mX^m $.
 \end{center}   
 Thus $X^m$ is a $\{1,3^E\}$ inverse of $A^m$.\\
 $(ii)$ Let $Y = X^{m+1}A^m$. Then applying Proposition \ref{lem3.2}, we have
\begin{eqnarray*}
YAY &=& X^{m+1}A^{m+1}X^{m+1}A^m=X^{m+1}AXA^m=X^m(XAX)A^m=X^{m+1}A^m=Y,
\end{eqnarray*}
\begin{center}
 $AY =AX^{m+1}A^m
     =AX^2X^{m-1}A^m
     =X^mA^m=X^mXA^{m+1}
     =X^{m+1}A^{m}A
     =YA$, and  
\end{center} 
  \begin{eqnarray*}
   YA^{k+1}&=&A^{k+1}Y=A^{k+1}X^{m+1}A^{m}=A^kX^mA^m=\cdots=X^{m-k}A^mX^mA^m\\
   &=&X^{m-k}A^m=X^{m-(k+1)}A^{m-1}=\cdots=XA^{k+1}=A^{k}.
 \end{eqnarray*}
 Hence $A^D=Y=X^{m+1}A^m$.
 \end{proof}
 The following result can be proved in similar manner for $F$-weighed dual core-EP inverse.
\begin{lemma}
Let $A,~F\in\cnn$ and $ind(A)=k$. For $m\geq k$, if $X=A^{F,\ep}$, then the following are true.
\begin{enumerate}
\item[(i)]  $X^m\in A\{1,4^F\}$.
\item[(ii)] $A^D = A^mX^{m+1}$.
\end{enumerate}
\end{lemma}
We now have the following characterization of the class of $\{1, 3^E\}$-inverse. 
\begin{lemma}\label{lem3.5}
Let $A,~E\in\cnn$ with $ind(A)=k$.  For $m\geq k$  if $X^m,Y^m\in A^m\{1,3^E\}$, then $A^mX^m = A^mY^m$. 
\end{lemma}
The proof is follows from Lemma \ref{lem27}.

Similarly, we obtain the following result for $F$-weighted dual core-EP inverse.  
\begin{lemma}
Let $A,~F\in\cnn$, and $ind(A)=k$.  For $m\geq k$ if $X^m,Y^m\in A^m\{1,4^F\}$, then $X^mA^m = Y^mA^m.$ 
\end{lemma}
The uniqueness of $E$-weighted core-EP inverse is discussed below.

 \begin{theorem}\label{eforeuni}
Let $A,~E\in\mathbb{C}^{n\times n}$ and $ind(A)=k$. If $A^{\ep,E}$ exist, then it is unique.
 \end{theorem}
 \begin{proof}
 Suppose there exist two $E$-weighted core-EP inverses $X$ and $Y$ of $A$. Let $m\geq k$. Then by Proposition \ref{lem3.2} and Lemma \ref{lem3.4}, we obtain
  \begin{center}
 $ X= XAX = X(A^m)X^m = X^2A^{m+1}X^m =\cdots= X^{m+1}A^{2m}X^m=A^DA^mX^m$,  
  \end{center}
  and $Y=A^DA^mY^m$. Hence by Lemma \ref{lem3.5}, we have $X=A^DA^mX^m=A^DA^mY^m=Y$.
 \end{proof}

Similarly, the uniqueness of $F$-weighted dual core-EP inverse can be verified. 
 \begin{theorem}
Let $A,~F\in\mathbb{C}^{n\times n}$ and $ind(A)=k$. If  $A^{F,\ep}$ exist, then it is unique. 
\end{theorem}

Now we discuss one of our important results, which gives a method of construction of the weighted core-EP inverse using a $\{1, 3\}$-inverse of matrix. The following result is easily follows from Lemma \ref{lem3.5} and Theorem \ref{eforeuni}.
 \begin{theorem}\label{thm3.12}
 Let $A,~E\in\mathbb{C}^{n\times n}$ and $ind(A)=k$. If $m$ is any positive integer satisfying $m\geq k$, and $A^m\{1,3^E\}$ is non-empty, then $A^{\ep,E}$ exists and $A^{\ep,E} = A^DA^mX$, where $X$ is a $\{1,3^E\}$ inverse of $A^m$.
 \end{theorem}
 Similarly, construction of the $F$-weighted dual core-EP inverse through  $\{1, 4^F\}$-inverse of a matrix is presented below. 
 
 \begin{theorem}\label{thm3.13}
  Let $A,~F\in\mathbb{C}^{n\times n}$ and $ind(A)=k$. 
  If $m$ is any positive integer satisfying $m\geq k$, and $A^m\{1,4^F\}$ is non-empty, then $A^{F,\ep}$ exists and $A^{F,\ep} = YA^mA^D$, where $Y$ is a $\{1,4^F\}$ inverse of $A^m$.
\end{theorem}

 In view of Theorem \ref{thm3.12} and \ref{thm3.13}, we state the following as a corollary for construction of weighted core-EP inverse.  

 \begin{corollary}\label{cor3.14}
  Let $A,~E,~F\in\mathbb{C}^{n\times n}$ and $ind(A)=k$. For $m\geq k$, if $(A^m)^\dagger_{E,F}$ exists, then  
  \begin{enumerate}
  \item[(i)] $A^{\ep,E}=A^DA^m(A^m)^{\dagger}_{E,I}$.
      \item[(ii)] $AA^{\ep,E}=A^mX=A^m(A^m)^{\dagger}_{E,I}$, $X\in A^m\{1,3^E\}$.
      \item[(iii)] $A^{F,\ep}=(A^m)^{\dagger}_{I,F}A^mA^D$.
      \item[(iv)] $A^{F,\ep}A=YA^m=(A^m)^{\dagger}_{I,F} A^m$,  $Y\in A^m\{1,4^F\}$.
  \end{enumerate} 
 \end{corollary}
 
In conjunction with Lemma \ref{lem2.6} and Corollary \ref{cor3.14}, we obtain the following representations for the weighted core-EP inverse through weighted core and weighted Moore-Penrose inverse. 

\begin{proposition}
 Let $E,F\in\mathbb{C}^{n\times n}$ be positive definite matrices and $A\in\cnn$ with $ind(A)=k$. Then \begin{enumerate}
      \item[(i)] $A^{\ep,E}=A^m(A^{m+1})^{\core,E}=A^m(A^{m+1})^{\dagger}_{E,I}$.
      \item[(ii)] $A^{F,\ep}=(A^{m+1})^{F,\core}A^m=(A^{m+1})^{\dagger}_{I,F}A^m$,
  \end{enumerate} 
  where $m$ is any positive integer satisfying  $m\geq k$.
\end{proposition}
\begin{proof}
$(i)$ From Lemma \ref{lemDraz} and \ref{lem2.6}, we get 
\begin{eqnarray}\label{eqn33}
        \nonumber
        	A^m(A^{m+1})^{\core,E}&=A^m\left(A^{m+1}\right)^{\#}A^{m+1}\left(A^{m+1}\right)^\dagger_{E,I}=A^m\left(A^D\right)^{m+1}A^{m+1}\left(A^{m+1}\right)^\dagger_{E,I} \\
		&= A^DA^{m+1}\left(A^{m+1}\right)^\dagger_{E,I}=A^{m}\left(A^{m+1}\right)^\dagger_{E,I}.
\end{eqnarray}
Using equation \eqref{eqn33} and Corollary \ref{cor3.14} $(i)$, we obtain
\begin{center}
  $A^{\ep,E}=A^m(A^{m+1})^{\core,E}=A^m(A^{m+1})^{\dagger}_{E,I}$.  
\end{center} 
$(ii)$ In similar way, we can show the result.
\end{proof}

We next discuss a necessary and sufficient condition in connection to the characterization of $\{1,3^E\}$ and $\{1,4^F\}$ inverses. 

 \begin{proposition}\label{prop3.15}
 Let $A,~E\in\mathbb{C}^{n\times n}$. Then $A\{1,3^E\}$  is non-empty if and only if $A = ZA^*EA$ for some $Z\in\cnn$. 
 \end{proposition}
 \begin{proof}
Let $Y\in A\{1,3^E\}$. Then 
 \begin{center}
  $A = AYA = E^{-1}(EAY)^*A =E^{-1}Y^*A^*EA = ZA^*EA $, where $Z = E^{-1}Y^*\in\cnn$.
\end{center}
Conversely, let $A = ZA^*EA$ and $Y=Z^*E$. Then      $AZ^* = ZA^*EAZ^*=(ZA^*EAZ^*)^*= (AZ^*)^*$. Further, $AYA = AZ^*EA=(AZ^*)^*EA = ZA^*EA = A$ and 
\begin{center}
   $(EAY)^*=(EAZ^*E)^* = EZA^*E = E(AZ^*)^*E = EAZ^*E=EAY$.  
\end{center}
  Thus $Y$ is an $\{1,3^E\}$ inverse of $A$.
 \end{proof}

 \begin{proposition}\label{prop3.18}
 Let $A,~F\in\cnn$. Then $A\{1,4^F\}$  is non-empty if and only if $A = AF^{-1}A^*X$ for some $X\in\cnn$. 
 \end{proposition}
 
 In view of the above Proposition \ref{prop3.15}, we obtain the following necessary and sufficient condition for  $\{1,3^E\}$ inverse.

  \begin{lemma}\label{lem3.17}
   Let $A,~E\in\cnn$. Then $A = X(A^*)^2EA$  for some $X\in\cnn$ if and only if $A = YA^*EA = A^2Z$ for some $Y, Z\in\cnn$.
 \end{lemma}
 \begin{proof}
 Let  $A = X(A^*)^2EA$. Then by taking $Y=XA^*$, we have $A= YA^*EA$. Applying Proposition \ref{prop3.15}, we have $Y^*E=AX^*E\in A\{1,3^E\}$. Thus 
 \begin{center}
     $A =  A(AX^*E)A =  A^2X^*EA = A^2Z$, where $Z=X^*EA\in\cnn$.
 \end{center}
 Conversely, let $A = YA^*EA = A^2Z$ for some $Y, Z\in\cnn$. Then $A = YA^*EA = Y(A^2Z)^*EA = YZ^*(A^2)^*EA  = X(A^2)^*EA$, where $X=YZ^*$.
 \end{proof}
 At the same time, we get the following result for 
 $\{1,4^F\}$ inverse through the Proposition \ref{prop3.18}.

 \begin{lemma}
  Let $A,~F\in\cnn$. Then $A = AF^{-1}(A^*)^2S$  for some $S\in\cnn$ if and only if $A = AF^{-1}A^*U= VA^2$ for some $U,V\in\cnn$.
 \end{lemma}
 In view of Theorem \ref{thm3.12}, Proposition \ref{prop3.15}, and Lemma \ref{lem3.17}, we have the following  sufficient condition for $E$-weighted core-EP inverse.

 \begin{theorem}
 Let $A,~E\in\mathbb{C}^{n\times n}$ and $ind(A)=k$. If $A^k = X((A^k)^*)^2EA^k$  for some $X\in\cnn$, then $A^{\ep,E}=A^DA^{2k}X^*E$. 
 \end{theorem}
 \begin{proof}
 Let $A^k = X((A^k)^*)^2EA^k$ for some $X\in\cnn$. Then by Lemma \ref{lem3.17}, we obtain $A^k=Y(A^k)^*EA^k$ for some $Y=X(A^k)^*\in\cnn$. Using Proposition \ref{prop3.15}, we have $A^kX^*E\in A^k\{1,3^E\}$. Thus by applying Theorem \ref{thm3.12}, we get $A^{\ep,E}=A^DA^{2k}X^*E$.
 \end{proof}

In view of $F$-weighted dual core-EP inverse can be written as follow. 
 
 \begin{theorem}
 Let $A,~F\in\mathbb{C}^{n\times n}$ and $ind(A)=k$. If $A^k = A^kF^{-1}((A^k)^*)^2Z$  for some $Z\in\cnn$, then $A^{F,\ep} = F^{-1}Z^*A^{2k}A^D$.
 \end{theorem}
 
We, next discuss the existence of the power of an $E$-weighted core-EP inverse.

 \begin{theorem}
  Let $A,~E\in\mathbb{C}^{n\times n}$ and $ind(A)=k (\in\mathbb{N})$.  Then $A^{\ep,E}$ exists if and only if $(A^k)^{\core,E}$ exists. Moreover, $A^{\ep,E} = A^{k-1}(A^k)^{\core,E}$ and $(A^k)^{\core,E} = (A^{\ep,E})^k$.
 \end{theorem}

 \begin{proof}
Let $X=A^{\ep,E}$ and $Y=X^k$. Then by Proposition \ref{lem3.2} $(i)$ and Lemma \ref{lem3.4} $(ii)$,
\begin{center}
    $A^k = A^DA^{k+1} = (X^{k+1}A^k)A^{k+1} =  X^k(XA^{k+1})A^k = X^k(A^k)^2=Y(A^k)^2$,
\end{center}
\begin{center}
  $A^kY^2=A^k(X^k)^2= (AX)X^k = AX^2X^{k-1} = X^k=Y$, and 
\end{center}
  \begin{center}
 $EA^kY=EA^kX^k = EAX = (EAX)^* = (EA^kX^k)^* =(EA^kY)^*$.
 \end{center}
 Thus $(A^k)^{\core,E} = Y=X^k=(A^{\ep,E})^k$.\\
 Conversely, let $Y=(A^k)^{\core,E}$ and $X = A^{k-1}Y$. Then by Lemma \ref{lem2.6}, we obtain
 \begin{center}
 $XA^{k+1}=A^{k-1}(A^k)^{\core,E}A^{k+1}= A^{k-1}(A^k)^{\#}A^k(A^k)^{(1,3^E)} A^{k+1} = A^{k-1}(A^k)^{\#}A^{k+1} = A^k$.  
 \end{center}
 Applying Lemma \ref{lem2.5}, we have  
 \begin{eqnarray*}
AX^2&=&A^k(A^k)^{\core,E}A^{k-1}(A^k)^{\core,E}=A^k(A^k)^{\core,E}A^{k-1}A^k\left((A^k)^{\core,E}\right)^2=A^{k-1}A^k\left((A^k)^{\core,E}\right)^2\\
&=&A^{k-1}(A^k)^{\core,E} =X.    
 \end{eqnarray*}
Further, $EAX = EA^k(A^k)^{\core,E}=(EA^k(A^k)^{E,\core})^*=(EAX)^*$. Hence 
\begin{center}
    $A^{\ep,E}=X=A^{k-1}Y=A^{k-1}\left(A^k\right)^{\core,E}$. \qedhere
\end{center}
\end{proof}

 Similarly, we present the following result for $F$-weighted dual core-EP inverse of a matrix.

 \begin{theorem}
  Let $A,~F\in\mathbb{C}^{n\times n}$ and $ind(A)=k(\in\mathbb{N})$. Then $A^{F,\ep}$ exists if and only if $(A^k)^{F,\core}$ exists. Moreover, $A^{F,\ep} =(A^k)^{F,\core} A^{k-1}$ and $(A^k)^{F,\core} = (A^{F,\ep})^k$.
 \end{theorem}

The power of $E$-weighted core-EP inverse and $E$-weighted core-EP inverse of power can be switched without changing the result.
\begin{theorem}\label{thm4.7}
Let $A,~E\in\mathbb{C}^{n\times n}$ and $ind(A)=k(\in\mathbb{N}$). Then for any positive integer $l$, $\left(A^{l}\right)^{\ep,E}$ exists if and only if  $A^{\ep,E}$ exists. In particular,   $(A^{l})^{\ep,E} = (A^{\ep,E})^{l}$ and $A^{\ep,E} = A^{l-1} (A^{l})^{\ep,E}.$
 \end{theorem}
 \begin{proof}
 Let  $A^{\ep,E} =   X.$ Then $  XA^{k+1} = A^k,$ $AX^2 =   X$, $(EAX)^* = EAX$. Let $m$ be a positive integer such that $0 \leq lm-k <l.$ Using Proposition \ref{lem3.2}, we have
 \begin{equation*}
 \aligned
   &(A^{l})^m= A^{lm} = A^{k}A^{lm-k} =   X A^{k+1} A^{lm-k}=  X A^{lm+1}=  X (A^{l})^m A \\
   &\qquad \ \ =   X^2 (A^{l})^mA^2=\cdots=  X^l (A^{l})^m A^l=  X^l (A^{l})^{m+1};\\
  &A^{l} (  X^{l})^2 = (A^{l}   X^{l})   X^{l} =AXX^{l} = ( A X^2)X^{l-1} =  X X^{l-1} =   X^{l};\\
   &(E A^{l}   X^{l})^* = (EAX)^* = EAX = EA^{l}X^{l}.
   \endaligned
  \end{equation*}
Hence $( A^{l})^{\ep,E} =   X^{l} = ( A^{\ep,E})^{l}$ with $ind( A^l)\leq m.$

Conversely, let $ind( A^l)= m$ and    $Y=(A^{l})^{\ep,E}.$ It is enough to show that $ X :=  A^{l-1} Y$ is the $E$-weighted core-EP inverse of $ A.$ Using Proposition \ref{lem3.2} $(i)$ and Lemma \ref{lem3.4} $(ii)$, we have 
   \begin{equation*}
   \aligned
     X  A^{lm+1}&=  A^{l-1} Y  A^{lm+1} =  A^{l-1}  A^l Y^2  A^{lm+1}=  A^{l-1}  A^{2l} YY^2  A^{lm+1}\\
  &=\cdots= A^{l-1} ( A^{ml} Y Y^m)  A^{lm+1}
    =  A^{lm+l-1} Y^{m+1}  A^{lm}  A =  A^{lm+l-1} Y^{m+1} ( A^{l})^{m}  A\\
    &= A^{lm+l-1} ( A^{l})^D A= ( A^{l})^D  A^{lm+l}
    =( A^{l})^D ( A^{l})^2  A^{lm-l}\\
    &= A^{l}  A^{lm-l} =  A^{lm},
    \endaligned
   \end{equation*}
     \begin{equation*}
   \aligned
   A X^2 &=  A A^{l-1} Y  A^{l-1} Y =  A^{l} Y  A^{l-1} Y=  A^{l} Y  A^{l-1} ( A^l Y^2)=A^{l} Y  A^{l-1} \left(( A^{l})^{m+1} Y^{m+2}\right)\\
   &=  A^{l} \left(Y ( A^{l})^{m+1}\right)  A^{l-1} Y^{m+2}= A^{l}  A^{lm}  A^{l-1} Y^{m+2}= A^{l-1} \left(( A^{l})^{m+1} Y^{m+2}\right)\\
   &= A^{l-1} Y =   X,
  \endaligned
 \end{equation*}
and 
 $$\left(EAX\right)^*= \left(EA A^{l-1}Y\right)^* = \left( EA^{l} Y\right)^*=EA^{l}Y = EAX.$$
 Thus $X$ is the $E$-weighted core-EP inverse of $A$ with $ind( A)\leq lm$. Hence $ A^{\ep,E}=  X= A^{l-1} Y= A^{l-1} \left( A^l\right)^{\ep,E}$.
 \end{proof}

We present the following result for $F$-weighted dual core-EP inverse.
 
 \begin{theorem}\label{thm4.8}
Let $A,~F\in\mathbb{C}^{n\times n}$ and $ind(A)=k(\in\mathbb{N})$. Then for any positive integer $l$, Then  $\left(A^{l}\right)^{F,\ep}$ exists if and only if $A^{F,\ep}$ exists. In particular,   $(A^{l})^{F,\ep} = (A^{F,\ep})^{l}$ and $A^{F,\ep} =  (A^{l})^{F,\ep}A^{l-1}.$
 \end{theorem}
 
It is well known that $(A^{-1})^{-1}= A$. However, $E$-weighted core-EP inverse is not following this property in general, i.e., $(A^{\ep,E})^{\ep,E} \neq A$. The following example confirm it. 
\begin{example}\rm
Let $A=
    \begin{pmatrix}
    -1 & 4 & -5\\
    1 & -4 & 5\\
    1 & -2 & 3\\
    \end{pmatrix}$ and $E=
    \begin{pmatrix}
    34/25 & 0 & 3/5\\
    0 & 1 & 0\\ 
    3/5 & 0 & 2\\
    \end{pmatrix}$. It is easy to verify $ind(A)=2$, $A^D=\begin{pmatrix}
    0 & 5/4 & -5/4\\
    0 & -5/4 & 5/4\\
    0 & -3/4 & 3/4\\
    \end{pmatrix}$, and 
    $X=\begin{pmatrix}
    0  &  0 &  0\\       
      -5/236  & 5/236   & 3/236\\   
       5/236    & -5/236  & -3/236 \\
       \end{pmatrix}$ is an $\{1,3^E\}$-inverse of $A^2$. Hence by Theorem \ref{thm3.12}, we obtain
       \begin{center}
          $A^{\ep,E}=A^DA^2X=\begin{pmatrix}
       -25/118   &      25/118    &     15/118  \\
      25/118    &    -25/118   &     -15/118 \\  
      15/118    &    -15/118   &      -9/118  \\ 
       \end{pmatrix}$. 
       \end{center}
       
       Similarly, we can calculate  
       \begin{center}
          $(A^{\ep,E})^{\ep,E}  =\begin{pmatrix}
       -50/59   &       50/59  &        30/59  \\  
      50/59      &   -50/59  &       -30/59  \\  
      30/59      &   -30/59 &        -18/59\\    
       \end{pmatrix}\neq A$. 
       \end{center}
       
           \end{example}
The following theorem represents the properties of $E$-weighted core-EP inverse.
\begin{theorem}\label{epofep}
 Let $A,~E\in\mathbb{C}^{n\times n}$ and $ind(A)=k$.  {If $A^{\ep,E}$ exists, then }$( A^{\ep,E})^{\ep,E}=  A^2  A^{\ep,E}$.
   \end{theorem}
 \begin{proof}
 If $k=0$, then the result is trivial. For $k=1$, the result follows from  Theorem \ref{thm2.4}, i.e.,
 \begin{equation*}
 {( A^{\core,E})}^{\core,E}=  A^2   A^\dagger=    A^2   A^{\core,E}.
 \end{equation*}
 Let $k \geq 2$ and $X=  A^2  A^{\ep,E}$. Now $$X\left( A^{\ep,E}\right)^{k+1}= A^2  \left( A^{\ep,E}\right)^{k+2}= A  A   \left( A^{\ep,E}\right)^2   \left( A^{\ep,E}\right)^{k}= A   \left( A^{\ep,E}\right)^2   \left( A^{\ep,E}\right)^{k-1}=  \left( A^{\ep,E}\right)^{k},$$
 \begin{eqnarray*}
    A^{\ep,E} X^2&=&   A^{\ep,E} ( A^2A^{\ep,E})^2 =   A^{\ep,E} A^2A^{\ep,E}A^2 A^{\ep,E}=A^{\ep,E}A^2   A^{\ep,E}A^2A\left( A^{\ep,E}\right)^2 \\
 &=&   A^{\ep,E}A^2    A^{\ep,E}  A^3   \left( A^{\ep,E}\right)^2
 =\cdots= A^{\ep,E} A^2A^{\ep,E}A^{k+1}   \left( A^{\ep,E}\right)^{k}\\
 &=&  A^{\ep,E} A^2 A^{k}   \left( A^{\ep,E}\right)^{k} =   A^{\ep,E} A^2A   A^{\ep,E}=  A^{\ep,E}  A^3   A^{\ep,E}=A^{\ep,E}A^2A  A^{\ep,E}\\
&=&A^{\ep,E} A^{k+1}   \left( A^{\ep,E}\right)^{k-1}=  A^{k}   \left( A^{\ep,E}\right)^{k-1} =  A^{k}  A^{\ep,E}   \left( A^{\ep,E}\right)^{k-2}\\
 &=&  A^{k+1}   \left( A^{\ep,E}\right)^{k} =  A A^{k}  \left( A^{\ep,E}\right)^{k} =  A  A  A^{\ep,E}\\
 &=&  A^2    A^{\ep,E} = X,\mbox{ and }
  \end{eqnarray*}
  \begin{eqnarray*}
    \left( EA^{\ep,E}X\right)^*&=& ( EA^{\ep,E}A^2 A^{\ep,E})^* = \left( EA^{\ep,E} A^3   \left( A^{\ep,E}\right)^2\right)^*\\
    &=& \left( EA^{\ep,E} A^{k+1}   \left( A^{\ep,E}\right)^{k}\right)^*=\left(E A^{k}   \left( A^{\ep,E}\right)^{k}\right)^* = \left(E A A^{\ep,E}\right)^*\\
    &=&  EA  A^{\ep,E} =  EA^{\ep,E} A^{k+1}   \left( A^{\ep,E}\right)^{k}\\
   &=& EA^{\ep,E}X.
  \end{eqnarray*} 
  Hence $\left( A^{\ep,E}\right)^{\ep,E}=X= A^2 A^{\ep,E}$.
 \end{proof}
 
   \begin{corollary}
  Let $A,~E\in\mathbb{C}^{n\times n}$ and  {$ind{(A)}=k$. If $A^{\ep,E}$ exists, then }
   $\left(( A^{\ep,E})^{\ep,E}\right)^{\ep,E}=  A^{\ep,E}$.
 \end{corollary}
 \begin{proof}
   Let $  B=  A^{\ep,E}.$ Using Theorem \ref{thm3.4} $(ii)$ and Theorem \ref{epofep}, we obtain
\begin{equation*}
\aligned
 \left(( A^{\ep,E})^{\ep,E}\right)^{\ep,E}&=\left(  B^{\ep,E}\right)^{\ep,E}=  B^2   B^{\ep,E}=\left( A^{\ep,E}\right)^2 \left( A^{\ep,E}\right)^{\ep,E}=\left( A^{\ep,E}\right)^2 \left( A^2  A^{\ep,E}\right)\\
 &=\left( A^{\ep,E}\right)^2  A \left( A  A^{\ep,E}\right)=\left( A^{\ep,E}\right)^2  A \left( A^k ( A^{\ep,E})^k\right)\\
 &= A^{\ep,E}  A^{\ep,E}  A^{k+1}  ( A^{\ep,E})^k= A^{\ep,E}  A^{k} ( A^{\ep,E})^k= A^{\ep,E}  A  A^{\ep,E}\\
 &= A^{\ep,E}. \qedhere
 \endaligned
  \end{equation*}
 \end{proof}

 In terms of $F$-weighted dual core-EP inverse, we obtain the following result. 
 \begin{theorem}
 Let $A,~F\in\mathbb{C}^{n\times n}$ and $ind(A)=k$.  {If $A^{F,\ep}$ exists, then } $( A^{F,\ep})^{F,\ep}=  A^{F,\ep}A^2$.
   \end{theorem}
 Similarly, for weighted dual core-EP inverse we have the following result.
 \begin{corollary}
  Let $A,~F\in\mathbb{C}^{n\times n}$ and $ind(A)=k$.  {If $A^{F,\ep}$ exists, then } $\left(( A^{F,\ep})^{F,\ep}\right)^{F,\ep}=  A^{F,\ep}$.
 \end{corollary}
The sufficient condition for the existence of weighted Moore-Penrose inverse is presented below. 
\begin{proposition}\label{prop3.25}
 Let $A,~E,~ F\in\mathbb{C}^{n\times n}$. If there exists $X,Y\in\cnn$ such that $AXA = A = AYA,~(EAX)^* = EAX \mbox{ and } (FYA)^*=FYA$, then $A^{\dagger}_{E,F} = YAX$.
\end{proposition}
\begin{proof}
Let $Z=YAX$. Then the result is obtained from the following verification:
\begin{center}
    $AZA= AYAXA = AXA=A$,
\end{center}
\begin{center}
    $ZAZ = YAXAYAX = YAYAX = YAX = Z$,
\end{center}
\begin{equation*}
  (EAZ)^* = EAYAX =EAZ, \mbox{~~and~~} (FZA)^* = FYAXA = FZA.  \qedhere
\end{equation*}
\end{proof}
A necessary and sufficient condition for the existence of weighted Moore-Penrose inverse is discussed, as follows. 
\begin{theorem}
Let $A,~E,~F\in\mathbb{C}^{n\times n}$. Then $A^\dagger_{E,F}$ exists if and only if $AF^{-1}A^*EAY=A=ZAF^{-1}A^*EA$ for some $Y\in\cnn$ and $Z\in\cnn$.
\end{theorem}
\begin{proof} 
 Let $X=A^\dagger_{E,F}$. Then 
\begin{eqnarray*}
A&=&AXA = AF^{-1}(FXA)^*=AF^{-1}A^*X^*F
 = AF^{-1}A^*X^*A^*X^*F\\
 &=&AF^{-1}A^*(E^{-1}EAX)^*X^*F
 = AF^{-1}A^*EAXE^{-1}X^*F = AF^{-1}A^*EAY,
 \end{eqnarray*}
 $Y =XE^{-1}X^*F$. Further, $A^* = Y^*A^*EAF^{-1}A^*=(EAY)^*AF^{-1}A^*$. This yields
\begin{center}
 $((EAY)^*A)^* =A^*EAY= (EAY)^*AF^{-1}A^*EAY=(EAY)^*A$.
\end{center}
Now 
\begin{eqnarray*}
A&=&AF^{-1}A^*EAY= AF^{-1}(EAY)^*A = AF^{-1}Y^*(A^*)EA= AF^{-1}(Y^*)^2A^*EAF^{-1}A^*EA\\
&=& ZAF^{-1}A^*EA, \mbox{ where } Z = AF^{-1}(Y^*)^2A^*E.
\end{eqnarray*}
Conversely, let $A =ZAF^{-1}A^*EA$, for some $Z\in\cnn$. Then by Proposition \ref{prop3.15},  $F^{-1}(ZA)^*E=F^{-1}(EZA)^*$  is a $\{1,3^E\}$ inverse of $A$. Similarly, from $A = AF^{-1}A^*EAY$, we obtain $F^{-1}(EAY)^*$ is a $\{1,4^F\}$ inverse of $A$. Hence by Proposition \ref{prop3.25},
\begin{equation*}
  A^{\dagger}_{E,F} = F^{-1}(EAY)^*AF^{-1}(EZA)^*. \qedhere   
\end{equation*}
\end{proof}

Let us recall the additive properties of the Drazin inverse \cite{BehNaJ, ben} of matrices.
 \begin{theorem}\label{ad-draz}
    Let $A,B\in\cnn$. If $ AB=O=  BA,$ then 
    \begin{center}
         $\left( A+  B\right)^D= A^D+  B^D$ and $\left( A-  B\right)^D= A^D-B^D$.
    \end{center}
   
    \end{theorem}
Using the above fact, we discuss a result on the additive property of the $E$-weighted core-EP inverse.
 \begin{theorem}
  Let $A,~B,~E\in\cnn$  with $A^*EB=O$, and  $AB=O=BA$.  {If $A^{\ep,E}$ and $B^{\ep,E}$ exists,} then 
      $\left( A+  B\right)^{\ep,E}= A^{\ep,E}+  B^{\ep,E}$ and  $\left( A- B\right)^{\ep,E}= A^{\ep,E}- B^{\ep,E}$.
   \end{theorem}

 \begin{proof}
    Let $A^*EB=O$, and $ AB=O=BA$. Then we have 
   $$
 \aligned
 & B^*EA=(A^*EB)^*=O,~AB^{\ep,E}= AB \left(  B^{\ep,E}\right)^2=O,~B  A^{\ep,E}=  B  A \left( A^{\ep,E}\right)^2=O,\\
 &   B^{\ep,E}A=  B^{\ep,E}   B   B^{\ep,E}  A=  B^{\ep,E}E^{-1}(E   B   B^{\ep,E}) A= B^{\ep,E}E^{-1} \left( B^{\ep,E}\right)^*   B^*E  A=O,\\ 
 & A^{\ep,E}B= A^{\ep,E}  A  A^{\ep,E} B= A^{\ep,E}E^{-1} \left( A^{\ep,E}\right)^*  A^* E  B=O,\\ 
 & A^{\ep,E}   B^{\ep,E}= A^{\ep,E} B \left( B^{\ep,E}\right)^2=O,~ B^{\ep,E}  A^{\ep,E}=  B^{\ep,E} A\left(  A^{\ep,E}\right)^2=O.
\endaligned
$$
Let $k=\max\{ind(A), ind(B)\}.$ Using Lemma \ref{lem3.4} $(i)$, we obtain 
\begin{center}
  $ A^k \left( A^{\ep,E}\right)^k  A^k= A^k$ and $  B^k \left(  B^{\ep,E}\right)^k   B^k=  B^k.$      
\end{center}
Now
 \begin{equation*}
    \aligned
    ( A&+  B)^k \left(( A^{\ep,E})^k+(  B^{\ep,E})^k\right) ( A+  B)^k=( A^k+  B^k) \left(( A^{\ep,E})^k+(  B^{\ep,E})^k\right) ( A^k+  B^k)\\
    &=\left( A^k ( A^{\ep,E})^k+  B^k (  B^{\ep,E})^k\right) ( A^k+  B^k)=\left( A  A^{\ep,E}+  B   B^{\ep,E}\right) ( A^k+  B^k)\\
    &= A  A^{\ep,E}  A^k+  B   B^{\ep,E}   B^k= A^k ( A^{\ep,E})^k  A^k+  B^k (  B^{\ep,E})^k   B^k= A^k+  B^k,
    \endaligned
    \end{equation*}
and
    \begin{equation*}
    \aligned
    \left(E( A+  B)^k \left(( A^{\ep,E})^k+(  B^{\ep,E})^k\right)\right)^*&=\left(E A  A^{\ep,E}+  EB   B^{\ep,E}\right)^*\\
    &=\left(E A  A^{\ep,E}+E B   B^{\ep,E}\right)\\
    &=( A+  B)^k \left(( A^{\ep,E})^k+(  B^{\ep,E})^k\right).
    \endaligned
    \end{equation*}
Therefore,  $( A^{\ep,E})^k+(  B^{\ep,E})^k\in (A+  B)^k\{1,3^E\}.$ By Theorem \ref{thm3.12}, we have
\begin{equation*}
\aligned
( A+  B)^{\ep,E}&=( A+  B)^D ( A+  B)^k \left(( A^{\ep,E})^k+(  B^{\ep,E})^k\right)\\
&=( A^D+  B^D) ( A^k+  B^k) \left(( A^{\ep,E})^k+(  B^{\ep,E})^k\right)\\
&=( A^D  A^k+  B^D   B^k) \left(( A^{\ep,E})^k+(  B^{\ep,E})^k\right)\\
&=  A^D  A^k \left( A^{\ep,E}\right)^k+  B^D   B^k \left(  B^{\ep,E}\right)^k\\
&=  A^D  A^k \left( A^{k}\right)^{\core,E}+  B^D   B^k \left(  B^{k}\right)^{\core,E}= A^{\ep,E}+  B^{\ep,E}.
\endaligned
\end{equation*}
Similarly, we can verify the other identity. 
 \end{proof}
 \begin{theorem}
  Let $A,~B,~F\in\cnn$  with $ AF^{-1}B^*=O$, and $ AB=O=BA$.  {If $A^{F,\ep}$ and $B^{F,\ep}$ exists}, then     $\left( A+  B\right)^{F,\ep}= A^{F,\ep}+  B^{F,\ep}$ and $\left( A-  B\right)^{F,\ep}= A^{F,\ep}- B^{F,\ep}.$   
  \end{theorem}

\section{Star weighted core-EP and weighted core-EP star}
 In this section, we discuss the new class of matrices to solve matrix type of equations, i.e., star weighted core-EP and weighted core-EP star matrices. 

\begin{theorem}\label{thm44.1}
Let $A,~E\in\cnn$ and $ind(A)=k$.  {If $A^{\ep,E}$ exists,} then $X =A^*AA^{\ep,E}$ is the unique solution of the following  system of matrix equations
    \begin{equation}\label{eq4.1}
     X(A^{\dagger})^*X = X,~~XA^k=A^* A^k \mbox{ and } (A^{\dagger})^*X = AA^{\ep,E}.  
    \end{equation}
\end{theorem}
\begin{proof}
Let $X =A^*AA^{\ep,E}$. Then by Theorem \ref{thm3.4} and Lemma \ref{lem3.4}, we obtain
\begin{eqnarray*}
X(A^{\dagger})^*X &=& A^*AA^{\ep,E}(A^{\dagger})^*A^*AA^{\ep,E} = A^*AA^{\ep,E}AA^{\dagger}AA^{\ep,E}\\
&=& A^*AA^{\ep,E}AA^{\ep,E} = A^*AA^{\ep,E} = X,
\end{eqnarray*}
\begin{center}
    $XA^k = A^*AA^{\ep,E}A^k = A^*A^k(A^{\ep,E})^kA^k = A^*A^k$, and
\end{center}
\begin{eqnarray*}
(A^{\dagger}E)^*X &=& E(A^{\dagger})^*A^*AA^{\ep,E} = E(AA^{\dagger})^*AA^{\ep,E} =E AA^{\dagger}AA^{\ep,E} = EAA^{\ep,E}.
\end{eqnarray*}
Next we claim the uniqueness of $X$. Let $X$ and $Y$ be two solutions which satisfy  equation \eqref{eq4.1}. Then  \begin{align*}
X &=~ X(A^{\dagger})^*X = XE^{-1}(A^{\dagger}E)^*X=XAA^{\ep,E} = XA^k(A^{\ep,E})^k = A^*A^k(A^{\ep,E})^k\\
&=~ YA^k(A^{\ep,E})^k= YE^{-1}EAA^{\ep,E} =Y(A^{\dagger})^*Y = Y. \qedhere 
\end{align*}
\end{proof}
Similarly, we can prove for $F$-weighted dual core-EP inverse.
\begin{theorem}\label{thm44.2}
Let $A,~F\in\cnn$ and $ind(A)=k$.  {If $A^{F,\ep}$ exists}, then $Y =A^{F,\ep}AA^*$ is the unique solution of the following  system of matrix equations
    \begin{equation}\label{eq4.2}
     Y(A^{\dagger})^*Y = Y,~~A^kY=A^kA^*, \mbox{ and } Y(A^{\dagger})^* = A^{F,\ep}A. 
    \end{equation}
\end{theorem}
In view of Theorem \ref{thm44.1} and \ref{thm44.2}, we define the following special matrices.
\begin{definition}
Let $A,~E\in\cnn$ and $ind(A)=k$. The star $E$-weighted core-EP  matrix of $A$ is denoted as $A^{*,\ep,E}$ and defined by $A^{*,\ep,E}  = A^*AA^{\ep,E}$.
\end{definition}
\begin{definition}
Let $A,~F\in\cnn$ and $ind(A)=k$. The star $F$-weighted dual core-EP  matrix of $A$ is denoted as $A^{F,\ep,*}$ and defined by $A^{F,\ep,*}  = A^{F,\ep}AA^*$.
\end{definition}
We can observe that, both $A^{*,\ep,E}$, $A^{F,\ep,*}$ matrices are outer inverse of $(A^\dagger)^*$. Next we discuss further properties of these special matrices. 

\begin{theorem}
Let $A,~E\in\cnn$ and $ind(A)=k$.  {If $A^{\ep,E}$ exists,} then the following statements are equivalent:
\item[(i)] $X$ is the star $E$-weighted core-EP  matrix of $A$.
\item[(ii)]$XAA^{\ep,E} = X$ and $XA^k = A^*A^k$.
\item[(iii)]$A^{\dagger}AXAA^{\ep,E} = X$ and $(A^{\dagger})^*XA^k = A^k$.
\item[(iv)]$XAA^{\ep,E} = X$ and $XA = A^*AA^{\ep,E}A$.
\item[(v)]$XAA^{\ep,E} = X$ and $X(A^{\dagger})^* = A^*AA^{\ep,E}(A^{\dagger})^*$.
\item[(vi)]$A^{\dagger}AX = X$ and $(A^{\dagger})^*X = AA^{\ep,E}$.
\item[(vii)]$A^{\dagger}AX = X$and $A^{\dagger}(A^{\dagger})^*X =A^{\dagger}AA^{\ep,E}$.
\item[(viii)]$A^{\dagger}AX = X$and $AX = AA^*AA^{\ep,E}$.
\item[({ix})]$XA^{\ep,E}A(A^{\dagger})^*X = X$,~$A^{\ep,E}A(A^{\dagger})^*X = AA^{\ep,E}$, and $XA^{\ep,E}A(A^{\dagger})^* = A^*A^{\ep,E}A(A^{\dagger})^*$.
\item[({x})]$XA^{\ep,E}A(A^{\dagger})^*X = X$,~$A^{\ep,E}A(A^{\dagger})^*X = AA^{\ep,E}$, $XA^{\ep,E}A(A^{\dagger})^* = A^*A^{\ep,E}A(A^{\dagger})^*$, and $A^{\ep,E}A(A^{\dagger})^*X A^{\ep,E}A(A^{\dagger})^* =A^{\ep,E}A(A^{\dagger})^*$.
\end{theorem}

\begin{proof}
$(i) \Rightarrow (ii).$ The proof follows from Theorem \ref{thm44.1} and 
\begin{center}
$XAA^{\ep,E} = XE^{-1}EAA^{\ep,E}= XE^{-1}(A^{\dagger}E)^*X=X(A^{\dagger})^*X  =X$.
\end{center}
$(ii) \Rightarrow (i)$. Using Proposition \ref{lem3.2} $(i)$, we get 
\begin{center}
   $X = XAA^{\ep,E} = XA^k(A^{\ep,E})^k = A^*A^k(A^{\ep,E})^k = A^*AA^{\ep,E}$. 
\end{center}
Hence $X$ is the star $E$-weighted core-EP matrix of $A$.\\
$(i) \Rightarrow (iii)$. Let $X=A^*AA^{\ep,E}$. Then $(A^{\dagger})^*XA^k = (A^{\dagger})^*A^*A^k = AA^{\dagger}AA^{k-1} = A^k$.
From Proposition \ref{lem3.2} $(ii)$,  we have
\begin{eqnarray*}
A^{\dagger}AXAA^{\ep,E} &= &A^{\dagger}AA^*AA^{\ep,E}AA^{\ep,E} = (A^{\dagger}A)^*A^*AA^{\ep,E}=(AA^{\dagger}A)^*AA^{\ep,E}\\
&=& A^*AA^{\ep,E} = X.
\end{eqnarray*}
$(iii) \Rightarrow (iv)$. By Proposition \ref{lem3.2} $(ii)$,
$XAA^{\ep,E} = A^{\dagger}AXAA^{\ep,E}AA^{\ep,E} = A^{\dagger}AXAA^{\ep,E} = X$. Now 
\begin{center}
$XA = A^{\dagger}AXAA^{\ep,E}A=A^*(A^\dagger)^*XA^k(A^{\ep,E})^kA =A^*A^k(A^{\ep,E})^kA 
= A^*AA^{\ep,E}A$.    
\end{center}
$(iv) \Rightarrow (i)$. Observe that:  $X = XAA^{\ep,E} = A^*AA^{\ep,E}AA^{\ep,E} = A^*AA^{\ep,E}$.\\
$(i) \Rightarrow (v)$ This follows from the following verification:
\begin{center}
    $X = A^*AA^{\ep,E}=A^*AA^{\ep,E}AA^{\ep,E}=XAA^{\ep,E}$ and $X(A^{\dagger})^* = A^*AA^{\ep,E}(A^{\dagger})^* $.
\end{center}
$(v) \Rightarrow (i)$ Observe that: 
\begin{eqnarray*}
X &=& XAA^{\ep,E}=X(A^\dagger)^*A^*AA^{\ep,E}=A^*AA^{\ep,E}(A^{\dagger})^*A^*AA^{\ep,E}=A^*AA^{\ep,E}AA^{\ep,E}\\
&=&A^*AA^{\ep,E}.
\end{eqnarray*}
$(i) \Rightarrow (vi)$. Let $X=A^*AA^{\ep,E}$. Then $A^{\dagger}AX=(A^{\dagger}AA^*)AA^{\ep,E} =  A^*AA^{\ep,E} = X$ and 
$(A^{\dagger})^*X = (A^{\dagger})^*A^*AA^{\ep,E} = AA^{\dagger}AA^{\ep,E}=AA^{\ep,E}  $.\\
$(vi) \Rightarrow (i).$  It follows from $X =A^{\dagger}AX=  (A^{\dagger}A)^*X = A^*(A^{\dagger})^*X = A^*AA^{\ep,E}$.\\
$(i) \Leftrightarrow (vii).$ Similar to $(i) \Leftrightarrow(vi)$.\\
$(i) \Leftrightarrow  (x)$. It is sufficient to show  $(vii)\Rightarrow (i)$ which observed from 
\begin{center}
  $X = A^{\dagger}AX =A^{\dagger}AA^*AA^{\ep,E} = A^*AA^{\ep,E} $.  
\end{center}
$(i) \Rightarrow (viii)$. Let $X=A^*AA^{\ep,E}$. Then 
\begin{eqnarray*}
XA^{\ep,E}A(A^{\dagger})^*X&=& A^*AA^{\ep,E}A^{\ep,E}A(A^{\dagger})^*A^*AA^{\ep,E} =  A^*A(A^{\ep,E})^2A^2A^{\ep,E}\\
&=& A^*A^{\ep,E}A^2A^{\ep,E} = A^*A^{\ep,E}A^{k+1}(A^{\ep,E})^k = A^*A^k(A^{\ep,E})^k\\
&=& A^*AA^{\ep,E}=X,
\end{eqnarray*}
\begin{eqnarray*}
A^{\ep,E}A(A^{\dagger})^*X &=& A^{\ep,E}A(A^{\dagger})^*A^*AA^{\ep,E} = A^{\ep,E}A^2A^{\ep,E}=A^{\ep,E}A^{k+1}(A^{\ep,E})^k\\
&=& A^k(A^{\ep,E})^k  = AA^{\ep,E}, \mbox{ and }
\end{eqnarray*}
\begin{center}
$XA^{\ep,E}A(A^{\dagger})^* = A^*AA^{\ep,E}A^{\ep,E}A(A^{\dagger})^* =  A^*A^{\ep,E}A(A^{\dagger})^*$.
\end{center}
$(ix) \Rightarrow (x)$. This implication holds from the following expression
\begin{center}
$A^{\ep,E}A(A^{\dagger})^*X A^{\ep,E}A(A^{\dagger})^*=AA^{\ep,E}A^{\ep,E}A(A^{\dagger})^* = A(A^{\ep,E})^2A(A^{\dagger})^*=A^{\ep,E}A(A^{\dagger})^*.$
\end{center}
$(x) \Rightarrow (i)$ Using the hypothesis, we have 
\begin{equation*}
  X = XA^{\ep,E}A(A^{\dagger})^*X = A^*A^{\ep,E}A(A^{\dagger})^*X = A^*AA^{\ep,E}.  \qedhere 
\end{equation*}
\end{proof}

Similarly, the following result holds for $F$-weighted dual core-EP star matrices.
\begin{theorem}
Let $A,~F\in\cnn$ and $ind(A)=k$.  {If $A^{F,\ep}$ exists,} then following are equivalent:
\item[(i)] $Y$ is the $F$-weighted dual core-EP star matrix of $A.$
\item[(ii)]$A^{F,\ep}AY = Y$ and $A^kY = A^kA^*$.
\item[(iii)]$A^{F,\ep}AYAA^{\dagger} = Y$ and $A^kY(A^{\dagger})^* = A^k$.
\item[(iv)]$A^{F,\ep}AY = Y$ and $AY = AA^{F,\ep}AA^*$.
\item[(v)]$A^{F,\ep}AY = Y$ and $(A^{\dagger})^*Y = (A^{\dagger})^*A^{F,\ep}AA^*$.
\item[(vi)]$YAA^{\dagger} = Y$ and $Y(A^{\dagger})^* = A^{F,\ep}A$.
\item[(vii)]$YAA^{\dagger} = Y$and $Y(A^{\dagger})^*A^{\dagger} =A^{F,\ep}AA^{\dagger}$.
\item[(viii)]$YAA^{\dagger} = Y$and $YA = A^{F,\ep}AA^*A$.
\item[({ix})]$Y(A^{\dagger})^*AA^{F,\ep}Y = Y$,~$Y(A^{\dagger})^*AA^{F,\ep}= A^{F,\ep}A$, and $(A^{\dagger})^*AA^{F,\ep}Y =(A^{\dagger})^*A A^{F,\ep}A^*$.
\item[({x})]$Y(A^{\dagger})^*AA^{F,\ep}Y = Y$,~$Y(A^{\dagger})^*AA^{F,\ep}= A^{F,\ep}A$,  $(A^{\dagger})^*AA^{F,\ep}Y =(A^{\dagger})^*A A^{F,\ep}A^*$, and  $(A^{\dagger})^*AA^{F,\ep}Y(A^{\dagger})^*AA^{F,\ep} = (A^{\dagger})^*AA^{F,\ep}$.
\end{theorem}

An interesting property of star $E$-weighted core-EP matrix is established in next result.
\begin{theorem}
Let $A,~E\in\cnn$ and $ind(A)=k$.  {If $A^{\ep,E}$ exists, and} $X=A^{*,\ep,E}$, then
\begin{enumerate}
    \item[(i)] $(A^\dagger)^*X$ is a projector onto $R(A^k)$ along $N(A^{\ep,E})$.
    \item[(ii)] $X(A^\dagger)^*$ is a projector onto $R(A^*A^k)$ along $N(A^{\ep,E}(A^\dagger)^*)$.
    \item[(iii)] $\left[(A^\dagger)^*\right]^{(2)}_{R(A^*A^k),N(A^{\ep,E})}=X$.
\end{enumerate}
\end{theorem}
\begin{proof}
$(i)$ From $(A^\dagger)^*X=(A^\dagger)^*A^*AA^{\ep,E}=AA^{\ep,E}=A^k\left(A^{\ep,E}\right)^k$, we obtain
\begin{center}
    $\left((A^\dagger)^*X\right)^2=AA^{\ep,E}AA^{\ep,E}=AA^{\ep,E}=(A^\dagger)^*X$, $R\left((A^\dagger)^*X\right)\subseteq R(A^k)$,
\end{center}
and $N(A^{\ep,E})\subseteq N\left((A^\dagger)^*X\right)$.
Since $A^k=AA^\dagger A^{k}=(A^\dagger)^*A^*A^k=(A^\dagger)^*XA^k$, and $A^{\ep,E}=A^{\ep,E}AA^{\ep,E}=A^{\ep,E}(A^\dagger)^*X$, it follows that $R(A^k)\subseteq R\left((A^\dagger)^*X\right)$ and $N(A^\dagger)^*X)\subseteq N(A^{\ep,E})$. Thus $R(A^k)= R\left((A^\dagger)^*X\right)$ and $N(A^{\ep,E})= N\left((A^\dagger)^*X\right)$.\\
$(ii)$ Using part $(i)$, $(X(A^\dagger)^*)^2=XAA^{\ep,E}(A^\dagger)^*=X(A^\dagger)^*$. From $X(A^\dagger)^*=A^*AA^{\ep,E}(A^\dagger)^*=A^*A^k(A^{\ep,E})^k(A^\dagger)^*$, we have $R(X(A^\dagger)^*)\subseteq R(A^*A^k)$ and $N(A^{\ep,E}(A^\dagger)^*)\subseteq N(X(A^\dagger)^*)$. Now 
 $N(A^{\ep,E}(A^\dagger)^*)= N(X(A^\dagger)^*)$ and $R(X(A^\dagger)^*)= R(A^*A^k)$ follows from the identities: 
\begin{center}
$A^*A^k=XA^k=X(A^\dagger)^*XA^k$ and \\
    $A^{\ep,E}(A^\dagger)^*=A^{\ep,E}AA^{\ep,E}(A^\dagger)^*=A^{\ep,E}(A^\dagger)^*A^*AA^{\ep,E}(A^\dagger)^*=A^{\ep,E}(A^\dagger)^*X(A^\dagger)^*$.
    \end{center}
$(iii)$ It follows from Theorem \ref{thm44.1}, part $(i)$, and $(ii)$.
 \end{proof}
 Using the similar lines, we can verify the following result for $F$-dual core-EP star matrices.
 \begin{theorem}
Let $A,~F\in\cnn$ and $ind(A)=k$.  {If $A^{F,\ep}$ exists, and } $Y=A^{\ep,F,*}$, then
\begin{enumerate}
    \item[(i)] $Y(A^\dagger)^*$ is a projector onto $R(A^k)$ along $N(A^{F,\ep})$.
    \item[(ii)] $(A^\dagger)^*Y$ is a projector onto $R(A^kA^*)$ along $N((A^\dagger)^*A^{F,\ep})$.
    \item[(iii)] $\left[(A^\dagger)^*\right]^{(2)}_{R(A^k),N((A^\dagger)^*A^{F,\ep})}=Y$.
\end{enumerate}
\end{theorem}

\section{Conclusion}
The notion of $E$-weighted core-EP and $F$-weighted dual core-EP inverses are introduced. Then, a few characterizations of weighted core-EP inverse and its relation with other generalized inverses are established. In addition to these, star weighted core-EP and weighted core-EP star class of matrices are introduced. Based on these special matrices, a few characterizations have been discussed. We pose the following problems for further research on the proposed matrices. 

\begin{enumerate}
    \item[$\bullet$] To develop an iterative method for estimating 
the star $E$-weighted core-EP  and $F$-weighted dual core-EP  matrices.
\item[$\bullet$] To study perturbation results and its applications for the
star $E$-weighted core-EP  and $F$-weighted dual core-EP  matrices.
\item[$\bullet$] To investigate the reverse order law for these inverses.
\end{enumerate}

\medskip

\noindent{\bf{Acknowledgments}}\\
Ratikanta Behera is grateful to the Mohapatra Family Foundation and the College of Graduate Studies, University of Central Florida, Orlando, FL, USA, for their financial support for this research.

\end{document}